\documentclass[12pt]{amsart}

\usepackage{amsmath,amssymb,amsthm}
\usepackage[all]{xy}
\usepackage[colorlinks=true]{hyperref}
\usepackage{epic,eepic}
\usepackage{here}

\numberwithin{equation}{section}

\SelectTips{eu}{12}

\newtheorem{theorem}{Theorem}[section]
\newtheorem{proposition}[theorem]{Proposition}
\newtheorem{lemma}[theorem]{Lemma}
\newtheorem{corollary}[theorem]{Corollary}
\newtheorem{problem}[theorem]{Problem}

\theoremstyle{definition}

\theoremstyle{remark}

\newcommand{\Z}{\mathbb{Z}}
\newcommand{\Q}{\mathbb{Q}}

\newcommand{\C}{\mathbb{C}}
\newcommand{\ch}{\mathrm{ch}}

\title{$p$-local stable cohomological rigidity of quasitoric manifolds}

\author{Sho Hasui}
\address{Department of Mathematics, Kyoto University, Kyoto, 606-8502, Japan}
\email{s.hasui@math.kyoto-u.ac.jp}

\author{Daisuke Kishimoto}
\address{Department of Mathematics, Kyoto University, Kyoto, 606-8502, Japan}
\email{kishi@math.kyoto-u.ac.jp}

\date{\today}

\begin{document}

\maketitle

\baselineskip 18pt

\begin{abstract}
It is proved that if two quasitoric manifolds of dimension $\le 2p^2-4$ for a prime $p$ have isomorphic cohomology rings, then they have the same $p$-local stable homotopy type.
\end{abstract}

\section{Introduction}

A class $\mathcal{C}$ of spaces is called cohomologically rigid if any spaces in $\mathcal{C}$ having isomorphic cohomology rings are homeomorphic with each other. It is well known that cohomology rings do not distinguish closed manifolds up to homeomorphism (or even homotopy equivalence), so the class of all closed manifolds is not cohomologically rigid. But what can we say about the cohomological rigidity if we restrict to a class of manifolds with good symmetries? The manifolds that we consider in this paper are quasitoric manifolds which were introduced by Davis and Januszkiewicz \cite{DJ} as a topological counterpart of smooth projective toric varieties. Since their introduction, quasitoric manifolds have been prominent objects which produce fruitful interactions of algebra, combinatorics, geometry, and topology. Formally, a quasitoric manifold is defined by a $2n$-dimensional manifold $M$ with a locally standard $n$-dimensional torus, say $T^n$, action such that the orbit space $M/T^n$ is identified with a simple polytope as manifolds with corners, where a locally standard $T^n$-action means that it is locally a coordinatewise $T^n$-action on $\C^n$. We refer to \cite{BP} for details. 

The cohomological rigidity problem for quasitoric manifolds was originally posed by Masuda, where there is a good survey \cite{CMS2}. For several simple quasitoric manifolds, the cohomological rigidity problem was affirmatively solved as in  \cite{DJ,CMS1,CPS,H1,H2}, but their approaches are quite ad-hoc. So we would like to consider the cohomological rigidity probem for general quasitoric manifolds. In general, we can approach to the cohomological rigidity in two steps which are quite different in nature: the first step is to show that spaces in question with isomorphic cohomology rings are homotopy equivalent, and the second step is to convert the homotopy equivalences obtained in the first step into homeomorphisms. In this paper, we study the first step for quasitoric manifolds from homotopy theoretical point of view. We will actually consider the following problem.

\begin{problem}
Do quasitoric manifolds with isomorphic cohomology rings have the same $p$-local stable homotopy type?
\end{problem}

As a first step to attack this problem, the authors and Sato \cite{HKS} obtained the following result which is a consequence of the $p$-local splitting of $\Sigma M$ and $\Sigma N$ in \cite{HKS}; under the assumption of the theorem, the splitting shows that $\Sigma M_{(p)}$ and $\Sigma N_{(p)}$ are wedges of $p$-local spheres.

\begin{theorem}
[Hsaui, Kishimoto, and Sato \cite{HKS}]
\label{HKS}
If $M,N$ are quasitoric manifolds with the same betti numbers and $\dim M=\dim N<2p$, then $\Sigma M_{(p)}\simeq\Sigma N_{(p)}$.
\end{theorem}


This paper shows a much more general $p$-local stable cohomological rigidity of quasitoric manifolds by considering $K$-theory, where we do not employ the $p$-local stable splitting of quasitoric manifolds. We say that an isomorphism $\theta\colon H^*(X)\xrightarrow{\cong}H^*(Y)$ is $p$-locally realized by a stable map if there is a stable map $h\colon\Sigma^\infty Y_{(p)}\to\Sigma^\infty X_{(p)}$ which is $\theta\otimes\Z_{(p)}$ in cohomology with $\Z_{(p)}$-coefficient. Note in particular that $h$ is a $p$-local stable homotopy equivalence by the J.H.C. Whitehead theorem whenever $X,Y$ are CW-complexes. We now state our main result.

\begin{theorem}
\label{main}
Any cohomology isomorphism between quasitoric manifolds of dimension $\le 2p^2-4$ is $p$-locally realized by a stable map.
\end{theorem}

\begin{corollary}
If two quasitoric manifolds of dimension $\le 2p^2-4$ have isomorphic cohomology rings, then they have the same $p$-local stable homotopy type.
\end{corollary}

Hereafter, let $p$ denote an odd prime unless otherwise specified. The 2-primary case will be dealt with only at the end of this paper.




\section{Adams $e$-invariant}

In this section, we recall the definition of the (complex) Adams $e$-invariant and its properties, and generalize it to maps from an odd sphere into a CW-complex without odd dimensional cells, where we refer to \cite{A} for details. Let $\pi^S_*$ denote the stable homotopy groups of spheres. Take $f\in\pi_{2k-1}^S$. Then it is a map $f\colon S^{2n+2k-1}\to S^{2n}$ for $n$ large. We now consider the $K$-theory of the mapping cone of $f$. Since there is a homotopy cofibration $S^{2n}\to C_f\to S^{2n+2k}$, $K(C_f)$ is a free abelian group of rank 2, and we can choose generators $\xi,\eta$ of $K(C_f)$ such that $\ch(\xi)=u_{2n}+au_{2n+2k}$ and $\ch(\eta)=u_{2n+2k}$ for $a\in\Q$, where $\ch\colon K(X)\to H^*(X)\otimes\Q$ and $u_i$ denote the Chern character and a generator of $H^i(C_f)\cong\Z$ respectively. Then the assignment 
$$e\colon\pi_{2k-1}^S\to\Q/\Z,\quad f\mapsto[a]$$
turns out to be a well-defined homomorphism, which is the Adams $e$-invariant. The property of the complex Adams $e$-invariant that we are going to use is the following.

\begin{theorem}
[Adams {\cite[Example 12.8]{A}} and Toda {\cite[Theorem 4.15]{T}}]
\label{e-invariant}
The Adams $e$-invariant $e\colon\pi_{2k-1}^S\to\Q/\Z$ is injective for $k\le p^2-3$ when localized at the prime $p$.
\end{theorem}



We call a CW-complex consisting only of even dimensional cells evenly generated. We generalize the Adams $e$-invariant for maps from odd dimensional spheres into evenly generated CW-complexes. Let $X$ be a connected evenly generated finite CW-complex of dimension $2d$, and let $X^{(r)}$ denote its $r$-skeleton. We choose a basis of $K(X^{(2k)})$ called an admissible basis by induction on $k$: 
\begin{itemize}
\item Fix a basis $x_1^i,\ldots,x_{n_i}^i$ of $H^{2i}(X)$ for $i>0$. 
\item Choose a basis $\mathcal{B}^1:=\{\xi_1^1,\ldots,\xi_{n_1}^1\}$ of $\widetilde{K}(X^{(2)})$ satisfying $\ch(\xi_i^1)=x_i^1$. 
\item Choose a basis $\mathcal{B}^k:=\widehat{\mathcal{B}}^{k-1}\cup\{\xi_1^k,\ldots,\xi_{n_k}^k\}$ of $\widetilde{K}(X^{(2k)})$ such that $\widehat{\mathcal{B}}^{k-1}$ restricts to $\mathcal{B}^{k-1}$ and $\ch(\xi_i^k)=x_i^k$, where the element of $\widehat{\mathcal{B}}^{k-1}$ restricting to $\xi_j^i\in\mathcal{B}^{k-1}$ is denoted by $\xi_j^i$.
\end{itemize}
The following property of admissible bases is clear from the definition.

\begin{proposition}
\label{equiv-admissible}
Let $X,Y$ be connected evenly generated finite CW-complexes. For a homotopy equivalence $h\colon X\xrightarrow{\simeq}Y$ and an admissible basis $\mathcal{B}$ of $\widetilde{K}(Y)$, $h^*(\mathcal{B})$ is an admissible basis of $\widetilde{K}(X)$, where $h^*(\mathcal{B}):=\{h^*(\xi)\,\vert\,\xi\in\mathcal{B}\}$.
\end{proposition}

For a map $f\colon S^{2r-1}\to X$, we define a basis $\mathcal{B}^d(f)$ of $\widetilde{K}(C_f)$ from an admissible basis $\mathcal{B}^d$ of $\widetilde{K}(X)$ by $\mathcal{B}^d(f):=\widehat{\mathcal{B}}^d\cup\{\eta\}$ such that $\widehat{\mathcal{B}}^d$ restricts to $\mathcal{B}^d$ and $\ch(\eta)=u_{2r}$, where $u_{2r}$ represents the cell attached by $f$ and $\xi_j^i\in\widehat{\mathcal{B}}^d$ denotes the element restricting to $\xi_j^i\in\mathcal{B}^d$. We now define  $e(\mathcal{B}^d(f))^i_j\in\Q$ by
$$\ch(\xi_j^i)=e(\mathcal{B}^d(f))^i_ju_{2r}+\text{other terms}\in H^*(C_f)\otimes\Q$$
which is a generalization of the Adams $e$-invariant that we are going to use to detect the triviality of $f$. We observe basic properties of our generalization of the Adams $e$-invariant. Note that $X/X^{(2d-2)}\simeq\bigvee_{n_d}S^{2d}$ such that $j^\text{th}$ sphere $S^{2d}$ corresponds to the cohomology class $x_j^d$. Let $\pi_j$ be the composite $X\xrightarrow{\rm proj}X/X^{(2d-2)}\simeq\bigvee_{n_d}S^{2d}\to S^{2d}$, where the last arrow is the pinch map onto the $j^\text{th}$ sphere. By definition, we immediately have the following.

\begin{lemma}
\label{e-top}
For $r>d$, $e(\mathcal{B}^d(f))_j^d\equiv e(\pi_j\circ f)\mod 1$.
\end{lemma}

When $f$ deforms into the $2k$-skeleton $X^{(2k)}$, we can construct both $e(\mathcal{B}^k(f))_j^i$ and $e(\mathcal{B}^d(f))_j^i$ for $i\le k$ by regarding $f$ as a map into $X^{(2k)}$ and $X$, respectively. By construction, we have the following.

\begin{lemma}
\label{e-induction}
If $f$ deforms into $X^{(2k)}$, then $e(\mathcal{B}^k(f))_j^i=e(\mathcal{B}^d(f))_j^i$ for $i\le k$.
\end{lemma}

\begin{proposition}
\label{e-invariant-trivial}
If $e(\mathcal{B}^d(f))_j^i$ is an integer for all $i,j$ and $d\le p^2-2$, then the $p$-localization of $f$ is stably null homotopic.
\end{proposition}

\begin{proof}
Localize everything at the prime $p$, so we abbreviate the notation $-_{(p)}$ for the $p$-localization. By the cellular approximation theorem, $f$ deforms into $X^{(2r-2)}$, so we consider a map $f\colon S^{2r-1}\to X^{(2r-2)}$ for which we can assume the same condition on the generalized Adams $e$-invariant by Lemma \ref{e-induction}. Consider the composite
$$\bar{f}\colon S^{2r-1}\xrightarrow{f}X^{(2r-2)}\xrightarrow{\rm proj}X^{(2r-2)}/X^{(2r-4)}\simeq\bigvee_{n_{r-1}}S^{2r-2}$$
where the $j^\text{th}$ sphere in the last space corresponds to $x_j^{r-1}$. Let $\pi_j\colon\bigvee_{n_{r-1}}S^{2r-2}\to S^{2r-2}$ be the pinch map onto the $j^\text{th}$ sphere. Then by Lemma \ref{e-top} and the assumption, we have $e(\pi_j\circ\bar{f})\equiv 0\mod 1$, implying $\pi_j\circ\bar{f}$ is stably null homotopic by Theorem \ref{e-invariant}. Thus we obtain that $\bar{f}$ itself is stably null homotopic. Consider the exact sequence of the stable homotopy groups
$$\pi_{2r-1}^S(X^{(2r-4)})\to\pi_{2r-1}^S(X^{(2r-2)})\to\pi_{2r-1}^S(X^{(2r-2)}/X^{(2r-4)}).$$
Then $f$ belongs to the middle group and is mapped to $\bar{f}$ by the last arrow, so it deforms into $X^{(2r-4)}$ stably. Hence, to continue the induction, it suffices to consider a map $f\colon S^{2r-1}\to X^{(2r-4)}$ for which we can assume the same condition on the generalized Adams $e$-invariant by Lemma \ref{e-induction} as well. Thus by iterating this procedure, we obtain that $f$ deforms stably into $X^{(2k)}$ for any $k$, implying $f$ is stably null homotopic. Therefore the proof is completed.
\end{proof}

\section{Realization of cohomology isomorphisms and $K$-theory}

This section studies the $p$-local stable realizability of cohomology isomorphisms between evenly generated CW-complexes by using $K$-theory. Throughout this section, let $X_1,X_2$ be connected evenly generated finite CW-complexes. We say that $\theta\colon K(X_1)\to K(X_2)$ is a lift of $\bar{\theta}\colon H^*(X_1)\to H^*(X_2)$ if the equality
$$\ch\circ\theta=(\bar{\theta}\otimes\Q)\circ\ch$$
holds. For the rest of this section, we assume that there are isomorphisms 
$$\theta\colon K(X_1)\xrightarrow{\cong}K(X_2)\quad\text{and}\quad\bar{\theta}\colon H^*(X_1)\xrightarrow{\cong}H^*(X_2)$$
which are compatible with the Chern character. We consider the $p$-local realizability of the isomorphism $\theta$ by a stable homotopy equivalence between $X_1$ and $X_2$. We first observe induced maps of $\theta,\bar{\theta}$ on subcomplexes and their quotinets. 

\begin{proposition}
\label{theta-restriction}
Let $Y_i$ be a subcomplex of $X_i$ for $i=1,2$ such that $\bar{\theta}$ restricts to an isomorphism $\hat{\theta}\vert_{Y_1}\colon H^*(Y_1)\xrightarrow{\cong}H^*(Y_2)$. Then $\theta,\bar{\theta}$ induce
\begin{enumerate}
\item an isomorphism $\theta\vert_{Y_1}\colon K(Y_1)\xrightarrow{\cong}K(Y_2)$ which is a lift of $\bar{\theta}\vert_{Y_1}$, and
\item isomorphisms $\Theta\colon K(X_1/Y_1)\xrightarrow{\cong}K(X_2/Y_2)$ and $\overline{\Theta}\colon H^*(X_1/Y_1)\xrightarrow{\cong}H^*(X_2/Y_2)$ such that $\Theta$ is a lift of $\overline{\Theta}$.
\end{enumerate}
\end{proposition}

\begin{proof}
We first show (2). Note that $X_i/Y_i$ is an evenly generated CW-complex since so are $X_i,Y_i$ and $Y_i$ is a subcomplex of $X_i$. Then there is a commutative diagram of solid arrows
\begin{equation}
\label{exact-seq}
\xymatrix{0\ar[r]&H^*(X_1/Y_1)\ar@{-->}[d]^{\bar{\Theta}}_\cong\ar[r]&H^*(X_1)\ar[d]^{\bar{\theta}}_\cong\ar[r]&H^*(Y_1)\ar[d]^{\bar{\theta}\vert_{Y_1}}_\cong\ar[r]&0\\
0\ar[r]&H^*(X_2/Y_2)\ar[r]&H^*(X_2)\ar[r]&H^*(Y_2)\ar[r]&0}
\end{equation}
with exact rows, so we get a dotted isomorphism $\overline{\Theta}$. As well as \eqref{exact-seq}, there is a commutative diagram
$$\xymatrix{0\ar[r]&K(X_i/Y_i)\ar[r]\ar[d]^\ch&K(X_i)\ar[d]^\ch\ar[r]&K(Y_i)\ar[r]\ar[d]^\ch&0\\
0\ar[r]&H^*(X_i/Y_i)\otimes\Q\ar[r]&H^*(X_i)\otimes\Q\ar[r]&H^*(Y_i)\otimes\Q\ar[r]&0}$$
with exact rows. Notice that the Chern character $\ch\colon K(X_i)\to H^*(X_i)\otimes\Q$ is injective since $H^*(X_i)$ is a free abelian group. Then it follows that $K(X_i/Y_i)$ is the kernel of the composite $f_i\colon K(X_i)\xrightarrow{\ch}H^*(X_i)\otimes\Q\to H^*(Y_i)\otimes\Q$. So since there is a commutative diagram
$$\xymatrix{K(X_1)\ar[r]^{f_1}\ar[d]^{\theta}_\cong&H^*(Y_1)\otimes\Q\ar[d]^{\bar{\theta}\vert_{Y_1}\otimes\Q}_\cong\\
K(X_2)\ar[r]^{f_2}&H^*(Y_2)\otimes\Q,}$$
we get an injection $\Theta\colon K(X_1/Y_1)\rightarrow K(X_2/Y_2)$ which becomes an isomorphism after tensoring $\Q$. Since $K(X_i/Y_i)$ is a direct summand of the free abelian group $K(X_i)$, we conclude that $\Theta$ is an isomorphism. Moreover, by a straightforward diagram chasing, we see that $\Theta$ is a lift of $\overline{\Theta}$. Therefore the proof of (2) is done. We finally prove (1). There is a commutative diagram of solid arrows
$$\xymatrix{0\ar[r]&K(X_1/Y_1)\ar[d]^{\Theta}_\cong\ar[r]&K(X_1)\ar[d]^\theta_\cong\ar[r]&K(Y_1)\ar@{-->}[d]\ar[r]&0\\
0\ar[r]&K(X_2/Y_2)\ar[r]&K(X_2)\ar[r]&K(Y_2)\ar[r]&0}$$
with exact rows. Then there is a dotted arrow which makes the diagram commute and is an isomorphism. Therefore (1) is proved.
\end{proof}

The cases to which we apply Proposition \ref{theta-restriction} are:
\begin{enumerate}
\item $Y_i=X_i^{(2k)}$ for $i=1,2$, and
\item $Y_i$ is a subcomplex $X_i^{(2k)}\cup e_i$ for $i=1,2$ such that $\bar{\theta}$ sends the cohomology class of $e_1$ to that of $e_2$.
\end{enumerate}

We now prove the $p$-local realizability of $\bar{\theta}$ by a stable map.

\begin{theorem}
\label{realizability}
For $\dim X_1=\dim X_2\le 2p^2-4$, $\bar{\theta}$ is $p$-locally realized by a stable map.
\end{theorem}

\begin{proof}
We put $\dim X_1=X_2=2d$, and denote the induced maps in Proposition \ref{theta-restriction} by the same symbols $\theta,\bar{\theta}$. We prove the $p$-local realizability of $\bar{\theta}$ by a stable map inductively on skeleta. We assume all spaces and maps are stabilized and $p$-localized, so we omit the stabilization functor $\Sigma^\infty$ and the $p$-localization $-_{(p)}$. 

The case $k=1$ is trivial since the spaces are wedges of $S^2$ for which any self-maps in homology is realizable. We now assume $k>1$ and there is a stable map $h\colon X_2^{(2k-2)}\to X_1^{(2k-2)}$ such that $h^*=\bar{\theta}\otimes\Z_{(p)}$. By arranging $2k$-cells of $X_2$, we may assume that $\theta,\bar{\theta}$ induce the identity map on $X_i^{(2k)}/X_i^{(2k-2)}:=\bigvee_aS^{2k}$. Let $\varphi_i\colon\bigvee_aS^{2k-1}\to X_i^{(2k-2)}$ be the attaching map of the $2k$-dimensional cells of $X_i$, and let $\iota_\ell\colon S^{2k-1}\to\bigvee_aS^{2k-1}$ denote the inclusion of the $\ell^\text{th}$ sphere. Then by Proposition \ref{theta-restriction} there are commutative diagrams
$$\xymatrix{0\ar[r]&K(S^{2k})\ar@{=}[d]\ar[r]&K(C_{\varphi_1\circ\iota_\ell})\ar[d]^\theta\ar[r]&K(X_1^{(2k-2)})\ar[d]^\theta\ar[r]&0\\
0\ar[r]&K(S^{2k})\ar[r]&K(C_{\varphi_2\circ \iota_\ell})\ar[r]&K(X^{(2k-2)}_2)\ar[r]&0}$$
and
$$\xymatrix{0\ar[r]&H^*(S^{2k})\ar@{=}[d]\ar[r]&H^*(C_{\varphi_1\circ\iota_\ell})\ar[d]^{\bar\theta}\ar[r]&H^*(X_1^{(2k-2)})\ar[d]^{h^*=\bar{\theta}}\ar[r]&0\\
0\ar[r]&H^*(S^{2k})\ar[r]&H^*(C_{\varphi_2\circ\iota_\ell})\ar[r]&H^*(X_2^{(2k-2)})\ar[r]&0}$$
with exact rows which are compatible by the Chern character. Since the Chern characters on these two diagrams are injective, we see that $h^*=\theta$ in the first diagram. Then we obtain
\begin{equation}
\label{e-1}
e(\mathcal{B}^{k-1}(\varphi_1\circ\iota_\ell))^i_j=e(h^*(\mathcal{B}^{k-1})(\varphi_2\circ\iota_\ell))^i_j
\end{equation}
for any $i,j$, where $h^*(\mathcal{B}^{k-1})$ is the admissible basis of $\widetilde{K}(X_2^{(2k-2)})$ as in Proposition \ref{equiv-admissible}. On the other hand, it immediately follows from the definition of the generalized Adams $e$-invariant that
\begin{equation}
\label{e-2}
e(h^*(\mathcal{B}^{k-1})(\varphi_2\circ\iota_\ell))^i_j=e(\mathcal{B}^{k-1}(h\circ\varphi_2\circ\iota_\ell))^i_j
\end{equation}
for any $i,j$. We now consider a map
$$f:=\varphi_1-h\circ\varphi_2\colon\bigvee_aS^{2k-1}\to X_1^{(2k-2)}.$$ 
By definition of the generalized Adams $e$-invariant, we have
$$e(\mathcal{B}^{k-1}(f\circ\iota_\ell))^i_j=e(\mathcal{B}^{k-1}(\varphi_1\circ\iota_\ell))^i_j-e(\mathcal{B}^{k-1}(h\circ\varphi_2\circ\iota_\ell))^i_j$$
for any $i,j$, so by \eqref{e-1} and $\eqref{e-2}$ we obtain $e(\mathcal{B}^{k-1}(f\circ\iota_\ell))^i_j=0$ for any $i,j$. Then by Proposition \ref{e-invariant-trivial}, $\varphi_1$ and $h\circ\varphi_2$ are stably homotopic, implying that there is a stable map $\tilde{h}\colon X_2^{(2k)}\to X_1^{(2k)}$ satisfying a homotopy commutative diagram
$$\xymatrix{\bigvee_aS^{2k-1}\ar@{=}[d]\ar[r]^{\varphi_1}&X_1^{(2k-2)}\ar[r]&X_1^{(2k)}\\
\bigvee_aS^{2k-1}\ar[r]^{\varphi_1}&X_2^{(2k-2)}\ar[u]_h\ar[r]&X_2^{(2k)}\ar[u]_{\tilde{h}}.}$$
Therefore by the Puppe exact sequence, we see that $\tilde{h}$ realizes $\bar{\theta}$, completing the proof.
\end{proof}

\section{Proof of Theorem \ref{main}}

This section applies Theorem \ref{realizability} to quasitoric manifolds, and then proves Theorem \ref{main}. We recall from \cite{DJ} properties of quasitoric manifolds that we are going to use.

\begin{proposition}
[Davis and Januszkiewicz \cite{DJ}]
\label{qt}
For a quasitoric manifold $M$, the following hold:
\begin{enumerate}
\item $M$ is a connected evenly generated finite CW-complex;
\item $H^*(M)$ is generated by 2-dimensional elements.
\end{enumerate}
\end{proposition}

\begin{theorem}
\label{K-lift}
For quasitoric manifolds $M_1,M_2$, any isomorphism $\bar{\theta}\colon H^*(M_1)\xrightarrow{\cong}H^*(M_2)$ lifts to an isomorphism $\theta\colon K(M_1)\xrightarrow{\cong}K(M_2)$.
\end{theorem}

\begin{proof}
Let $x_1,\ldots,x_\ell$ be a basis of $H^2(M_1)$. Then $\bar{\theta}(x_1),\ldots,\bar{\theta}(x_\ell)$ is a basis of $H^2(M_2)$. Put $\rho_1:=x_1\times\cdots\times x_\ell\colon M_1\to(\C P^\infty)^\ell$ and $\rho_2:=\bar{\theta}(x_1)\times\cdots\times\bar{\theta}(x_\ell)\colon M_2\to(\C P^\infty)^\ell$. By definition, we have 
$$\rho_2^*=\bar{\theta}\circ\rho_1^*$$
in cohomology. By considering the induced map between the Atiyah-Hirzebruch spectral sequences, we see that $\rho_i^*\colon K((\C P^\ell))\to K(M_i)$ is surjective for $i=1,2$. Then in order to get a map $\theta\colon K(M_1)\to K(M_2)$, it is sufficient to show that $\mathrm{Ker}\,\rho_1^*\subset\mathrm{Ker}\,\rho_2^*$. For $x\in K((\C P^\infty)^\ell)$, we suppose $\rho_1^*(x)=0$. Then we have
$$0=(\theta\otimes\Q)\circ\ch(\rho_1^*(x))=(\theta\otimes\Q)\circ\rho_1^*(\ch(x))=\rho_2^*(\ch(x))=\ch(\rho_2^*(x)),$$
implying $\rho_2^*(x)=0$ since $\ch\colon K(M_2)\to H^*(M_2)\otimes\Q$ is injective by Proposition \ref{qt}. Then we get a map $\theta\colon K(M_1)\to K(M_2)$ such that $\theta(\rho_1^*(y))=\rho_2^*(y)$ for any $y\in K((\C P^\infty)^\ell)$. We have that $\theta$ is a lift of $\bar{\theta}$. Indeed, for any $y\in K((\C P^\infty)^\ell)$,
$$\ch(\theta(\rho_1^*(y)))=\ch(\rho_2^*(y))=\rho_2^*(\ch(y))=(\bar{\theta}\otimes\Q)\circ\rho_1^*(\ch(y))=(\bar{\theta}\otimes\Q)(\ch(\rho_1^*(y)))$$
where $\rho_1^*\colon K((\C P^\infty)^\ell)\to K(M_1)$ is surjective. It remains to show that $\theta$ is an isomorphism. Since $\rho_2^*\colon K((\C P^\infty)^\ell)\to K(M_2)$ is surjective, so is $\theta$. If $\theta(x)=0$ for $x\in K(M_1)$, we have
$$0=\ch(\theta(x))=(\bar{\theta}\otimes\Q)(\ch(x)),$$
implying $x=0$ since $\bar{\theta}\otimes\Q$ is an isomorphism and $\ch\colon K(M_1)\to H^*(M_1)\otimes\Q$ is injective. Thus $\theta$ is injective, completing the proof.
\end{proof}

\begin{proof}
[Proof of Theorem \ref{main}]
Combine Theorem \ref{realizability} and \ref{K-lift} when $p$ is odd. For $p=2$ we only need to consider the case $\dim M_1=\dim M_2=4$ since $\dim M_1=\dim M_2$ implies $M_1=M_2=S^2$. The case $\dim M_i=4$ is proved in \cite{DJ}. Here is an alternative proof: $M_i$ has the stable homotopy type of a wedge of $S^2$ and $S^4$ or $\C P^2$ which is distinguished by mod 2 cohomology together with the action of the Steenrod operation $\mathrm{Sq}^2$. By Proposition, \ref{qt}, $\bar{\theta}$ respects $\mathrm{Sq}^2$, and therefore the proof is completed.
\end{proof}

\end{document}